\documentclass{amsart}

\usepackage[utf8]{inputenc}
\usepackage{lmodern}
\usepackage{tikz,tkz-tab}
\usepackage{amsmath,amssymb,amsthm,bbm}
\usepackage[english]{babel}
\usepackage{verbatim}

\usepackage{verbatim}

\newcommand{\communique}{\leftrightarrow}
\newcommand{\communiqueS}{\stackrel{S}{\rightarrow}}

\newcommand{\N}{\mathbb{Z}_{+}}

\newcommand{\Z}{\mathbb{Z}}

\newcommand{\Zd}{\mathbb{Z}^d}

\newcommand{\R}{\mathbb{R}}
\newcommand{\Rd}{\mathbb{R}^d}

\renewcommand{\P}{\mathbb{P}}
\newcommand{\E}{\mathbb{E}\ }

\newcommand{\Ber}{\text{Ber}}

\renewcommand{\epsilon}{\varepsilon}
\renewcommand{\phi}{\varphi}
\renewcommand{\limsup}{\overline{\lim}}
\renewcommand{\liminf}{\underline{\lim}}

\newcommand{\miniop}[3]{%
\renewcommand{\arraystretch}{0.6}
\begin{array}{c}
{\scriptstyle #1}\\
#2\\
{\scriptstyle #3}
\end{array}
\renewcommand{\arraystretch}{1}}

\newcommand{\1}{1\hspace{-1.3mm}1}

\newcommand{\pcfleche}{\overrightarrow{p_c}}

\begin{document}

\newtheorem{theorem}{Theorem}[section]
\newtheorem{conjecture}[theorem]{Conjecture}
\newcommand{\titre}{\title}

\newtheorem{lemma}[theorem]{Lemma}
\newtheorem{lemme}[theorem]{Lemma}
\newtheorem{defi}[theorem]{Definition}
\newtheorem{coro}[theorem]{Corollary}
\newtheorem{rem}[theorem]{Remark}
\newtheorem{prop}[theorem]{Proposition}
\newtheorem{theo}[theorem]{Theorem}


\titre[Percolation on oriented graphs]{Percolation and first-passage percolation on oriented graphs}
\date{\today}


\author{Olivier Garet}
\author{R{\'e}gine Marchand}
\address{Institut \'Elie Cartan Nancy (math{\'e}matiques)\\
Universit{\'e} de Lorraine\\
Campus Scientifique, BP 239 \\
54506 Vandoeuvre-l{\`e}s-Nancy  Cedex France}
\email{Olivier.Garet@univ-lorraine.fr}

\address{Institut \'Elie Cartan Nancy (math{\'e}matiques)\\
Universit{\'e} de Lorraine\\
Campus Scientifique, BP 239 \\
54506 Vandoeuvre-l{\`e}s-Nancy  Cedex France}
\email{Regine.Marchand@univ-lorraine.fr}

\def\motsclefs{Percolation, first-passage percolation, sharp transition.}

\subjclass[2000]{60K35, 82B43.} 

\begin{abstract}
   We give the first properties of independent Bernoulli percolation, for oriented graphs on the set of vertices $\Zd$ that are translation-invariant and may contain loops.
  We exhibit some examples showing that the critical probability for the existence of an infinite cluster  may be direction-dependent.
  Then, we prove that the phase transition in a given direction is sharp, and study the links between percolation and first-passage percolation on these oriented graphs.
\end{abstract}

\maketitle 

\newcommand{\Edod}{\overrightarrow{\mathbb{E}}^d}

In percolation and directed percolation on the cubic lattice $\Zd$, infinite clusters do not have the same geometry. In the unoriented setting, as soon as the opening parameter $p$ exceeds the critical value $p_c(d)$ for the existence of an infinite cluster, one can build an infinite path in any given direction. 
On the contrary, in the oriented setting, clusters starting from the origin only live in the first quadrant; more precisely, when the opening parameter $p$ exceeds  the critical value $\overrightarrow{p_c}(d)$,  a deterministic cone gives the directions in which infinite paths are found. 

Between these two models, it seems natural to ask what may happen for percolation for oriented graphs, on the set of vertices $\Zd$, whose connections do not forbid any direction, or in other words, for oriented graphs that contain loops.

In the present paper, we first exhibit one example of such an oriented graph,  where every direction is permitted, but such that we observe two phase transitions: if $p$ is small, there there exists no infinite path, then when $p$ increases there is a phase where infinite paths exist but not in any direction (as in classical supercritical oriented percolation), and finally, when $p$ is large enough, infinite paths can grow in any direction (as in classical supercritical unoriented percolation). 

Then, coming back to the general framework, we give some properties of percolation on oriented graphs on $\Zd$ that give an echo to some standard results for unoriented percolation, with a particular attention to the links between  oriented percolation and first-passage oriented percolation.





\section{The framework and one example}
We deal here with an oriented graph whose vertices are the elements of $\Zd$, and whose edges are the couples $(x,y)$ such that $y-x$ belongs to a given finite set denoted by $\text{Dir}$. Hence, if $E$ denotes the set of edges, one has
$$E=\{(x,y)\in\Zd\times\Zd:\; y-x\in \text{Dir}\}.$$

For a given parameter $p \in(0,1)$, we endow the set $\Omega=\{0,1\}^E$ with the Bernoulli product $\P_p=\Ber(p)^{\otimes E}$: under this probability measure, the edges are independently open (state $1$) with probability $p$ or closed (state $0$) with probability $1-p$, and we are interested in the connectivity properties of the random graph $G(\omega)$ whose edges are the ones that are open in $\omega$.

For $x\in\Zd$, we denote by $C_+(x)$ the set of points that can be reached from $x$ by a path in the random graph $G$, \emph{i.e.} the points $y$ such that
there exists a sequence $(x_0,\dots,x_n)$ with $x_0=x$, $x_n=y$ and
$(x_i,x_{i+1})\in E$ for each $i\in\{0,\dots,n-1\}$.

For $u\in\Rd \backslash \{0\}$, we define $$D_u(x)=\sup_{y\in C_+(x)} \langle y-x,u\rangle.$$
The field $(D_u(x))_{x\in\Zd}$ is stationary and ergodic.
We set 
$$\theta_u(p)=\P_p(D_u(0)=+\infty)\text{ and }p_c(u)=\inf\{p>0: \;\theta_u(p)>0\}.$$
The quantity $D_u(x)$ measures the extension of the oriented open cluster issued from $x$ in direction $u$ and $p_c(u)$ is the critical parameter for the existence of an oriented open cluster that is unbounded in direction $u$. Note however that $D_u(x)=+\infty$ does not imply the existence of infinitely many points  of $C_+(x)$ close to the half-line $\R_+ u$.

\subsection*{An example}

We take here $d=2$, we fix some positive integer $M$ and we choose
\begin{align*}
  \text{Dir}=\{(0,-1);(-M,1),(-M+1,1),\dots,(-1,1),(0,1),(1,1),\dots,(M,1)\}.
\end{align*}  
In other words, the only allowed communications are the following: for all $x,x',y \in \Z$, 
\begin{itemize}
\item $(x,y)\to (x',y+1)$ if $|x-x'|\le M$
\item $(x,y)\to (x,y-1)$
\end{itemize}
Let us denote by $(e_1,e_2)$ the canonical basis for $\R^2$: with this set of edges, we give an advantage to direction $e_2$ when compared to direction $-e_2$. 

We first observe that for $M$ large enough, there exist values for the opening parameter $p$ such that there is percolation in direction $e_2$ but not in direction $-e_2$:

\begin{theo} Denote by $\pcfleche(2)$ the critical value for classical oriented percolation on $\N^2$. 
\begin{itemize}
\item For $M \ge 1$, $\displaystyle \inf_{u \in  \Rd \backslash \{0\}} p_c(u) \ge \frac1{2M+2}$ and $\displaystyle \sup_{u   \in\Rd \backslash \{0\}} p_c(u) \le \pcfleche(2)$.
\item For $M\ge 2$, $\displaystyle p_c(-e_2)\ge\frac1{2\sqrt{2M+1}}$.
\item For $M\ge 5$,  $\displaystyle p_c(e_2)\le  1-(1-\pcfleche(2))^{2/M}<\frac{-2\log( 1-\pcfleche(2))}{M}\le  \frac{2 \log 3}M$.
\end{itemize}
Particularly, for $M\ge 37$, $p_c(e_2)<p_c(-e_2).$
\end{theo}

\begin{proof}
$\bullet$ The mean number of self-avoiding open paths starting from $(0,0)$ with length $n$ is  at most $((2M+2)p)^n$. Thus  if $p<\frac1{2M+2}$, the number of self-avoiding open paths is integrable and thus almost surely finite, and there is no percolation at all. When $p>\pcfleche(2)$,  restricting $  \text{Dir}$ to $\{(0,-1), (1,1)\}$, then to $\{(0,-1), (-1,1)\}$, then to $\{(1,1), (-1,1)\}$, we obtain three copies of the standard oriented percolation in $\N^2$, and thus three percolation cones: it is then easy to see that for any $u\in  \Rd \backslash \{0\}$, $\P(D_u(0)=+\infty)>0$. 

$\bullet$   For a fixed integer $\ell$, the graph $(\Z^2,E)$ contains exactly  ${2\ell+n\choose \ell}(2M+1)^{\ell}$ paths from  $(0,0)$ to $\Z\times\{-n\}$ that contains $\ell$ steps upwards and $\ell+n$ steps downwards.
Then, the mean number of open self-avoiding paths from $(0,0)$ to the line $y=-n$ is no more that
\begin{align*}
&\sum_{\ell=0}^{+\infty} {2\ell+n\choose \ell} (2M+1)^{\ell} p^{2\ell+n}
\le \sum_{\ell=0}^{+\infty}(2M+1)^{\ell} (2p)^{2\ell+n}= \frac{(2p)^n}{1-4p^2(2M+1)},
\end{align*}
as soon as $4p^2(2M+1)<1$.
It follows  that for $p<\frac1{2\sqrt{2M+1}}$, the number of self-avoiding paths from  $(0,0)$ to  $\{(x,y)\in\Z^2;y\le 0\}$ is integrable, therefore it is almost surely finite. This gives the first inequality.

$\bullet$ For the last inequality, we build a dynamic independent directed percolation from bloc events with length $M/2$ that partition the horizontal lines. Remember that $M\ge 5$.
The probability that a given point $(x,y)$ in the segment$(\frac{M}2 \overline{x}+[-M/4,M/4))\times \{y\}$
    can be linked to some point in $ (\frac{M}2 (\overline{x}+1)+[-M/4,M/4))\times y+1\}$
is larger than  $1-(1-p)^{M/2}$. So is the probability that
        one can link this point to some point in
         $(\frac{M}2 (\overline{x}-1)+[-M/4,M/4))\times \{y+1\}$.
 Hence, we built a dynamic percolation of blocks in the spirit of Grimmett and Marstrand~\cite{Grimmett-Marstrand} (see also Grimmett~\cite{grimmett-book}), that stochastically dominates an independent directed bond percolation on $\Z^2$, with  parameter $1-(1-p)^{M/2}$.
Then, percolation in direction $e_2$  is possible as soon as $1-(1-p)^{M/2}>\pcfleche(2)$, whence
\begin{align*}
  p_c(e_2)&\le 1-\exp \left(\frac{2}M\log(1-\pcfleche(2)) \right)<-\frac{2}M\log(1-\pcfleche(2))\le \frac{2\log 3}M,
  \end{align*}
where the last inequality comes from Liggett's bound~\cite{MR1359822}: $\pcfleche(2)\le 2/3$.
The desired result follows.
\end{proof}

\begin{figure}  
  \begin{tabular}{cc}
    \frame{\includegraphics[scale=0.05]{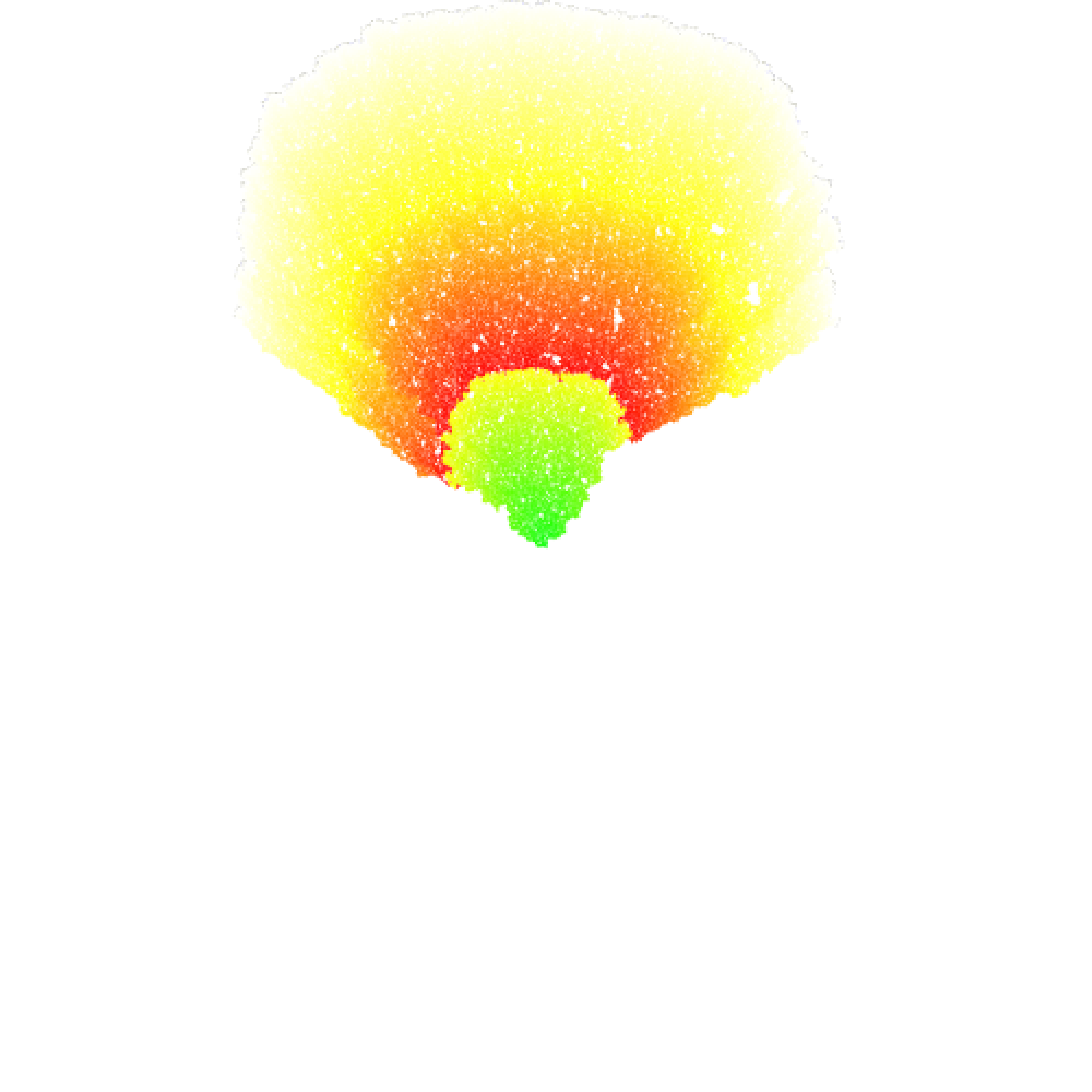}}&\frame{\includegraphics[scale=0.05]{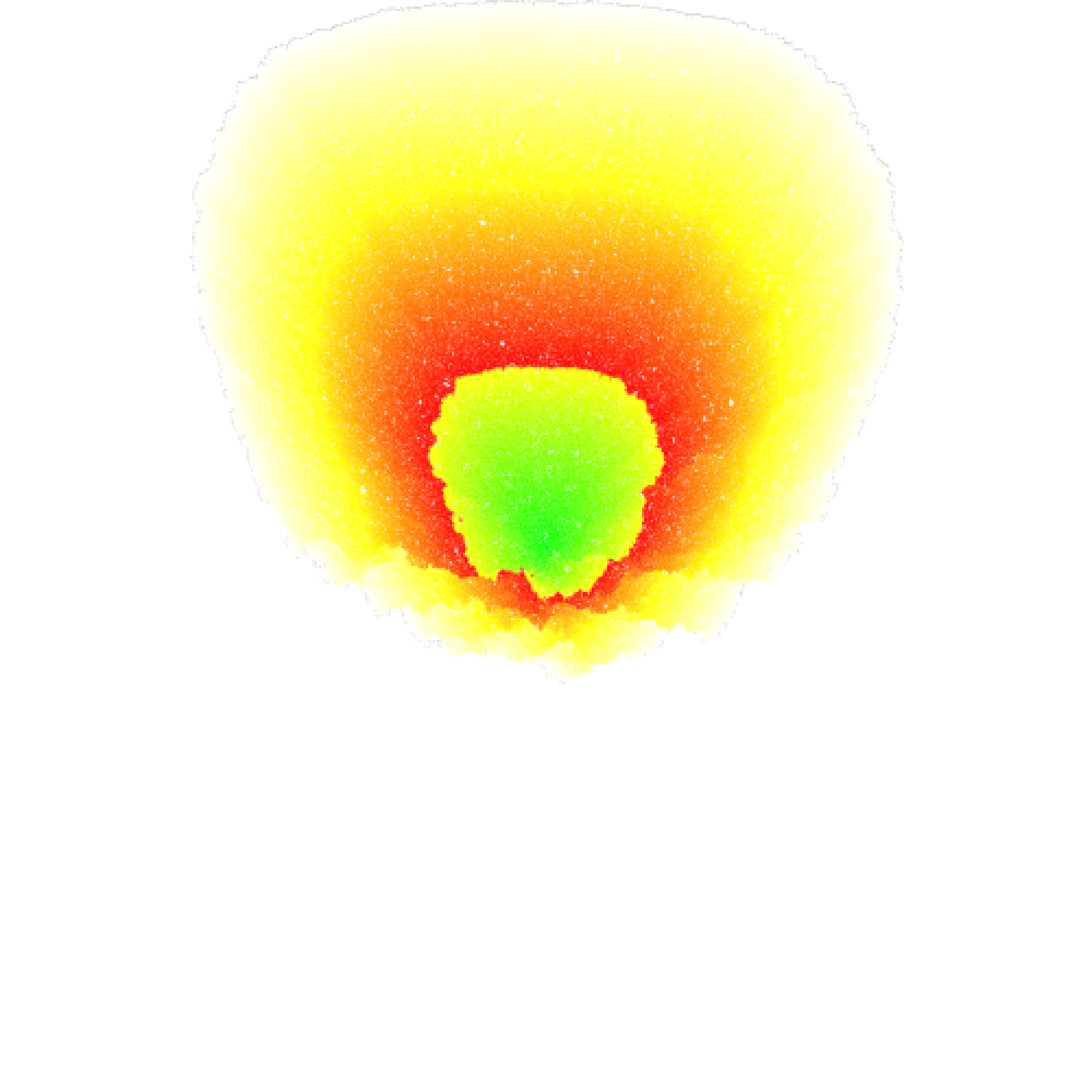}}\\
   \end{tabular} 
\caption{Oriented percolation with $M=1$, $p=0.51$ on the left and  $p=0.55$ on the right. The pictures are centered at the origin. The points are colored accordingly to their distance to the origin. The coloring is performed by the Dijkstra algorithm until one hits the border.}
\end{figure}

\section{A sharp percolation transition}

We now come back to our general framework.
Let $\Psi:\Zd\to\R$ be a subadditive function, i.e. such that for any $x,y \in \Zd$, 
$\Psi(x+y)\le\Psi(x)+\Psi(y)$.
  We define
  $$\forall x \in \Zd \quad r_\Psi(x)=\sup_{y\in C_+(x)}\Psi(y-x).$$
  The graph $(\Zd,E)$ being translation-invariant, the distribution of $r_\Psi(x)$ does not depend on $x$.

  If $A$, $B$ and $S$ are subsets of $\Zd$, the event $A\communiqueS B$ means that there exists a path $(x_0,\dots,x_n)$ with $x_0\in A$, $x_n\in B$, $x_i\in S$ for $i\in\{1,\dots,n-1\}$ and the bonds $(x_i,x_{i+1})$ are all open.
  
For $p\in[0,1]$ and $0\in S\subset \Zd$, we define
\begin{align}
\phi_p(S) & :=p \sum_{(x,y)\in \partial^+ S} \P_p (0\communiqueS  x),\text{where }\partial^+ S=E\cap(S\times (\Zd\backslash S))  \label{eq:16} \\
\tilde p_c(\Psi) &:= \sup\left\{
  \begin{array}{c}
  	p \in [0,1] : \; \text{there exists a set $S$ s.t. } 0\in S \subset \Zd \\
  	\text{with } \phi_p(S)<1 \text{ and } \sup_{S}\Psi <+\infty
  \end{array}\right\},  \label{eq:17} \\
    p_c(\Psi)&:=\sup\{p \in [0,1] : \; \P_p(r_\Psi(0)=\infty)=0\}\nonumber.
\end{align}
Note that in the above definition, the set $S$ may be infinite. Then, we have the following result:

\begin{theo}
\label{thm:perco} 
Fix $d\ge 2$. Let $\Psi:\Zd \to \R$ be  a subadditive function.
\begin{enumerate}
\item \label{item:1} 
For $p<\tilde p_c(\Psi)$, there exists $c=c(\Psi,p)>0$ such that for each $n\ge 1$,
$$\P_p(r_\Psi(0)\ge n)\le e^{-c n}.$$
\item\label{item:2} For $p>\tilde p_c(\Psi)$, $ \displaystyle \P_p(r_\Psi(0)=+\infty)\ge   \frac{p-\tilde p_c(\Psi)}{p(1-\tilde p_c(\Psi))}$.
\end{enumerate}
In particular,  \eqref{item:1} and~\eqref{item:2}   imply  that $\tilde p_c(\Psi)=p_c(\Psi)$.
\end{theo}
Note that $ \Psi_u(x)=\langle u,x\rangle$ is linear and thus subadditive, and, for this map, $ p_c(\Psi_u)=p_c(u)$. 

\begin{proof}
$\bullet$ At first, let us prove that~\eqref{item:1} and~\eqref{item:2}   imply $\tilde p_c(\Psi)=p_c(\Psi)$.
If $p<\tilde{p}_c(\Psi)$, then for each $n\ge 1$, we have $\P_p(r_\Psi(0)=+\infty)\le \P_p(r_\Psi(0)\ge n)\le e^{-cn}$; letting $n$ go to infinity, we get $\P(r_\Psi(0)=+\infty)=0$. So $\tilde{p}_c(\Psi)\le p_c(\Psi)$.
But~\eqref{item:2} implies that  $\P_p(r_\Psi(0)=+\infty)>0$ for $p> \tilde{p}_c(\Psi)$, thus  $\tilde p_c(\Psi)\ge p_c(\Psi)$.

\medskip
$\bullet$  Proof of~\eqref{item:1}: it is very similar to Duminil-Copin--Tassion~\cite{MR3477351,MR3605816,MR3783562}. Since it is short, we give it to stay self-contained.

Let $p<\tilde{p}_c(\Psi)$. By the very definition of $\tilde{p}_c(\Psi)$, we can find $S \subset \Zd$
that contains the origin and such that $\phi_p(S)<1$ and $\sup_S{\Psi}<+\infty$. Fix a positive integer $L\ge \sup_{S \cup  \text{Dir}} {\Psi}$.
We set 
$$\Lambda_n=\{x\in\Zd: \; \Psi(x)\le n\}.$$
 Thus, $\{r_\Psi(0)>n\}=\{0\to\Lambda_n^c\}$. For $k \ge 1$, an open path starting from $0$ and escaping from $\Lambda_{kL}$ eventually leaves $S$. Then,
\begin{align*}
  \{0\to  \Lambda_{2kL}^c\}=\miniop{}{\cup}{(x,y)\in\partial^+ S}\{0\communiqueS x,\omega_{(x,y)}=1,y \stackrel{S^c}{\rightarrow}\Lambda_{2kL}^c\}
    \end{align*}
By independence, we get
\begin{align*}
\P_p(r_\Psi(0)>2kL) &\le \miniop{}{\sum}{(x,y)\in\partial^+ S} \P_p(0\communiqueS x) \, p\,  \P_p(y \stackrel{S^c}{\rightarrow}\Lambda_{2kL}^c).
\end{align*}
Note that 
\begin{itemize}
\item If $(x,y)\in\partial^+ S$, then $\Psi(y)\le \Psi(x)+\max_{\text{Dir}} \Psi \le 2L$;
\item $\{ y \stackrel{S^c}{\rightarrow}\Lambda_{2kL}^c\}  \subset \{\exists z \in C_+(y): \; \Psi(z)>2kL \}$; 
\item thus if $(x,y)\in\partial^+ S$ and $z \in C_+(y)$ is such that $\Psi(z)>2kL$, then
$$\Psi(z-y) \ge \Psi(z)-\Psi(y) >2kL-2L =2(k-1)L.$$
\end{itemize}
We thus obtain
\begin{align*}
\P_p(r_\Psi(0)>2kL)  &\le \miniop{}{\sum}{(x,y)\in\partial^+ S}\P_p(0\communiqueS x) \, p \, \P_p(r_\Psi(y)> 2(k-1)L)\\
&\le\phi_p(S)\P_p(r_\Psi(0)> 2(k-1)L)
\end {align*}
It follows that $\P_p(r_\Psi(0)>2kL)\le\phi_p(S)^k$, which gives the desired result.
\end{proof}

\medskip
$\bullet$  Proof of~\eqref{item:2}. In Duminil-Copin--Tassion, the idea is to use the Russo inequality.
It is a  bit more tricky here, because the events $\{0\communique\partial\Lambda_n\}$, which  correspond to the exit of finite boxes in Duminil-Copin--Tassion, now depend on infinitely many bonds. The proof is cut into three lemmas. 

We begin with a lemma on a general graph.

\begin{lemme}\label{conditionne}
 Let $G=(V,E)$ be an oriented graph with $V$ finite or denumerable.
Let $\P$ denote a Bernoulli product on $\{0,1\}^E$. Let $X$ and $Y$ be disjoint subsets of $V$, with $\P(X\to Y)>0$.
For each  $S\subset V$, and each $(x,y)\in E$, we set
\begin{align*}
r_X^{(x,y)}(S) &=\1_{X \subset S} \1_{(x,y)\in\partial^+ S}\P(X\communiqueS x).
\end{align*}
We denote by $\mathcal{T}_Y$ the $\sigma$-field generated by the events $\{x\to Y\}$, for $x\in V$.
We denote by  $B_Y$ the random subset of $V$ composed by the points that are not linked to~$Y$.

Remember that $e\in E$ is said to be pivotal for an event $A\in\mathcal{B}(\{0,1\}^E)$ in the configuration $\omega\in \{0,1\}^E$ if $\1_A(0_e\omega_{E\backslash\{e\}})\ne \1_A(1_e\omega_{E\backslash\{e\}})$.

Then, for any $e \in E$,
$$\P(e\text{ pivotal for }X\to Y, X\not\to Y \; | \; \mathcal{T}_Y)=r_X^e(B_Y).$$
\end{lemme}

\begin{proof}[Proof of Lemma \ref{conditionne}]
Let us denote by $\Gamma$ the set of oriented paths in $Y^ c$ from a point in $Y^c$ to a point in $Y$. 
Then  the subsets $\cap_{\gamma \in A}\cap_{e \in\gamma}\{\omega_e=1\}$, for $A \subset \Gamma$,  form a $\pi$-system that generates $\mathcal{T}_Y$, so  it is enough to prove that for each  $A \subset \Gamma$, one has 
 \begin{align}
 \label{amtr}
 \P \left( \begin{array}{c} 
 e\text{ pivotal for }X\to Y, X\not\to Y, \\
 \forall \gamma\in A,\, \forall f \in \gamma, \, \omega_f=1
 \end{array} \right) &=\E \left( r_X^e(B_Y)\prod_{\gamma\in A}\prod_{f\in \gamma}\omega_f \right).
\end{align}
The quantities that appear on each side of~\eqref{amtr} are the limit of analogous quantities for a sequence of finite subgraphs of $G$.
 So, by dominated convergence, it is sufficient to prove~\eqref{amtr} for a finite graph. From now on, we assume that $G$ is finite.
 
 Decomposing on the (finite number of) possible values of $B_Y$, we thus only have to prove that for any subset $S$ of vertices such that  $X \subset S \subset Y^c$, 
\begin{align*}
&  \P \left( \begin{array}{c} 
 e\text{ pivotal for }X\to Y, \; B_Y=S \\
 \forall \gamma\in A,\, \forall f \in \gamma, \, \omega_f=1
 \end{array} \right) =\E \left( r_X^e(S) \1_{\{B_Y=S\}}\prod_{\gamma\in A}\prod_{f\in \gamma}\omega_f \right). 
  \end{align*}
Fix a set  $S$ such that  $X \subset S \subset Y^c$. Let us denote by
 \begin{align*}
E_1 & = \{ (x,y) \in E: \; x,y \in S\}, \\
E_2 &  =\partial^+ S = \{ (x,y) \in E: \; x \in S, \, y \in S^ c\}, \\
E_3 & =  \{ (x,y) \in E \backslash (E_1 \cup E_2): \;\exists (u,v) \in \partial^+ S, \,   \P_1(v  \stackrel{S^c}{\rightarrow} x,y \stackrel{S^c}{\rightarrow}Y)=1\}.
\end{align*}
Note that on the event $B_Y=S$, as $X \subset S$, pivotal edges for $X \to Y$ are necessarily in $E_2$ and that when $e \notin E_2$, both members vanish. 
The event $B_Y=S$ is measurable with respect to the states of the edges in $E_2 \cup E_3$, and implies that all edges in $E_2$ are closed.  Thus both members vanish  if $A \not\subset E_3$. Denote by $A_3$ the set of possible configurations  of edges in $E_3$ that correspond to $B_Y=S$. Finally, we thus have to prove that for any $S$ such that $X \subset S \subset Y^c$, for any $e=(x,y) \in E_2$, for any $\xi \in A_3$, 
 \begin{align*}
&  \P \left( \begin{array}{c} 
 e\text{ pivotal for }X\to Y, \\
 \forall f\in E_3,\omega_f=\xi_f, \; B_Y=S
 \end{array} \right) =\E \left( r_X^e(S) \1_{\{B_Y=S\}} \prod_{f\in E_3}\1_{\{\omega_f=\xi_f\}} \right). 
  \end{align*}
But now, by independence,
 \begin{align*}
&  \P \left( \begin{array}{c} 
 e\text{ pivotal for }X\to Y, \\
 \forall f\in E_3,\omega_f=\xi_f, \; B_Y=S
 \end{array} \right)
   = \P \left( \begin{array}{c} 
X\stackrel{S}{\rightarrow} x , \, \forall f \in E_2, \, \omega_f=0 \\
 \forall f\in E_3,\omega_f=\xi_f
 \end{array} \right) \\
= & \P(X\stackrel{S}{\rightarrow} x)   \P(\forall f \in E_2, \, \omega_f=0, \,
 \forall f\in E_3,\omega_f=\xi_f)   \\
 = & \P(X\stackrel{S}{\rightarrow} x)   \P(B_Y=S, \,
 \forall f\in E_3,\omega_f=\xi_f),
  \end{align*}
which is indeed the mean value of $ r_X^e(S) \1_{\{B_Y=S\}} \prod_{f\in E_3}\1_{\{\omega_f=\xi_f\}}$.
\end{proof}

We come back to the case of a graph on $\Zd$.
  
\begin{lemme}\label{proba}Let $p\in [0,1]$.  
 For every natural number $n$, we set $f_n(p)=\P_p(0 \to \Lambda_n^c)$ and
  $c_n=\inf_{S\subset\Lambda_n,0\in S}\phi_p(S)$.
  Then, for each $p\in[0,1[$.
  $$\liminf_{h\to 0^+}\frac{f_n(p+h)-f_n(p)}h\ge\frac1{p(1-p)} c_n (1-f_n(p)).$$
\end{lemme}

\begin{proof}[Proof of Lemma~\ref{proba}]
  The event $\{0\to \Lambda^c_n\}$ depends on infinitely many bonds, so one can not directly apply the Russo formula.
  However, since $\{0\to \Lambda_n^c\}$ is an increasing event, the following inequality is preserved (see for example Grimmett~\cite{grimmett-book}, page 43):
\begin{align*}
  \liminf_{h\to 0^+}\frac{f_n(p+h)-f_n(p)}h&\ge \sum_{e\in E}\P(e\text{ is pivotal for } 0 \to\Lambda^c_n ) \\
  &=\sum_{e\in E}\frac1{1-p}\P(e\text{ is pivotal for } 0 \to\Lambda^c_n,0\not \to\Lambda^c_n)  
\end{align*}
Now consider the random set  $S_n$ of points from which  $\Lambda_n^c$ can not be reached.
Note that $\{0\not\to\Lambda_n^c\}=\{0\in S_n\}$.
For each $S\subset\Zd$ and $(x,y)\in E$, we define the random variable
\begin{align*}r^{(x,y)}_p(S)&=\1_{(x,y)\in\partial^+ S}\P_p(0\communiqueS x).
  \end{align*}
Integrating the result of Lemma~\ref{conditionne}, we have for each $e\in E$: 
$$\P(e\text{ is pivotal for } 0 \to\Lambda^c_n, 0\not \to\Lambda^c_n)=\mathbb E_p\left( \1_{0 \in S_n} r^e_p(S_n)\right).$$
Then, we get
\begin{align*}
  \sum_{e\in E}\mathbb E_p\left( \1_{0 \in S_n}  r^e_p(S_n)\right)&=\mathbb E_p \left( \1_{ \{0\not \to\Lambda^c_n \} } \sum_{e\in E}r^e_p(S_n) \right) = \mathbb E_p \left(  \1_{ \{0\not \to\Lambda^c_n \} } \frac{\phi_p(S_n)}p \right) \\
  &\ge \mathbb E_p\left( \1_{ \{0\not \to\Lambda^c_n \} } \frac{c_n}p\right)=c_n\frac{1-f_n(p)}p,
\end{align*}
which gives the desired inequality.
\end{proof}

\begin{lemme}
\label{analyse}
Let $I\subset\R$ be an open interval of $\R$ and let $f$ and $h$ be real valued  functions defined on $I$ and such that
\begin{itemize}
\item $f$ is left upper semi-continuous on $I$ from the left: $\forall x \in I, \; f(x)\ge\liminf_{t\to x^-}f(t)$;
\item $h$ is continuous on $I$
\item For each $x\in I$
$$\liminf_{t\to 0^+}\frac{f(x+t)-f(x)}t\ge h(x).$$
\end{itemize}
Then, for any $a$ and $b$ in $I$ with $a\le b$, we have
$\displaystyle f(b)-f(a)\ge\int_a^b h(x)\ dx.$
\end{lemme}

\begin{proof}[Proof of Lemma \ref{analyse}]
Let $a,b\in I$ with $a<b$. We fix $\epsilon>0$ and define on $[a,b]$:
$F_{\epsilon}(x)=f(x)-\int_a^x h(t)\ dt +\epsilon x$.
It is sufficient to prove that  $F_{\epsilon}$ is non-decreasing for each $\epsilon>0$.
Indeed, it will imply that
$$f(b)-\int_a^b h(t)\ dt +\epsilon b=F_{\epsilon}(b)\ge F_{\epsilon}(a)=f(a)+\epsilon a,$$
which gives the lemma when  $\epsilon$ tends to $0$.

Let $x\in [a,b]$. By definition of $F_\epsilon$, 
$$\liminf_{t\to 0^+} \frac{F_{\epsilon}(x+t)-F_{\epsilon}(x)}t=\liminf_{t\to 0^+}\frac{f(x+t)-f(x)}t-h(x)+\epsilon\ge \epsilon.$$
So there exists $\eta_x>0$ such that for any $t \in (0,\eta_x)$,  $\frac{F_{\epsilon}(x+t)-F_{\epsilon}(x)}t\ge \epsilon/2\ge 0$.

Let $B=\{x\in [a,b]: \; F_{\epsilon}(x)<F_{\epsilon}(a)\}$.
Assume by contradiction that  $B\ne\varnothing$ and define $c=\inf B$. Let $(x_n)$ be a sequence in $B$ that tends to $c$.
By the previous observation, the inequality $F_{\epsilon}(x_n)\ge F_{\epsilon}(c)$ holds for $n$ large enough.
Since $x_n\in B$, by definition of $B$,   $F_\epsilon(a)>F_{\epsilon}(x_n)$. Thus $F_\epsilon(a)>F_{\epsilon}(c)$.

As $F_{\epsilon}$ is the sum of a function which is  upper semi-continuous from the left and of a continuous function,
it is still  upper semi-continuous from the left. So
$$F_{\epsilon}(c)\ge\liminf_{t\to c^-} F_{\epsilon}(t),$$
and by definition of $c$, $F_{\epsilon}(t)\ge F(a)$ for each $t\in ]a,c[$, so $F_{\epsilon}(c)\ge F_{\epsilon}(a)$. This brings a contradiction.
\end{proof}

\begin{proof}[End of the proof of Theorem \ref{thm:perco}: proof of~\eqref{item:2}]
Fix $p'\in ]\tilde{p}_c(\Psi),1[$ and define on $[0,1)$ the function $g(x)=-\log(1-x)$: it is non-decreasing and convex.

Let $p\in[ p', 1)$ and $h\in (0, 1-p)$:
$$\frac{g(f_n(p+h))-g(f_n(p))}{f_n(p+h)-f_n(p)}\frac{f_n(p+h)-f_n(p)}{h}\ge g'(f_n(p)) \frac{f_n(p+h)-f_n(p)}{h}.$$
With  Lemma \ref{proba} (note that as $p>\tilde{p}_c(\Psi)$, $c_n\ge 1$), we obtain that
$$\liminf_{h\to 0^+}\frac{g(f_n(p+h))-g(f_n(p))}h\ge\frac{c_n}{p(1-p)}\ge\frac{1}{p(1-p)}.$$
We can now apply Lemma~\ref{analyse} on $[p',1[$: as $f_n$ is non-increasing, $g\circ f_n$ is non-decreasing, so it is clearly  upper semi-continuous from the left: for any $p>p'$
$$g(f_n(p))\ge g(f_n(p))-g(f_n(p'))\ge\int_{p'}^{p} \frac{dx}{x(1-x)}=\log\frac{p(1-p')}{p'(1-p)}=g\left(\frac{p-p'}{p(1-p')}\right) .$$
It follows that $f_n(p)\ge \frac{p-p'}{p(1-p')}$, then, letting $p'$ tend to $\tilde{p}_c(\Psi)$, we get
$$f_n(p)\ge \frac{p-\tilde{p}_c(\Psi)}{p(1-\tilde{p}_c(\Psi))}.$$ 
Finally, we obtain~\eqref{item:2} by  letting $n$ go to
infinity.
\end{proof}

\section{Links with first-passage percolation}

\subsection{Percolation and first-passage percolation on the (unoriented) edges of~$\Zd$}

Consider first $\Zd$ endowed with the set $E_d$ of edges between nearest neighbors. In the first-passage percolation model, iid non negative and  integrable random variables $(t_e)_{e \in E_d}$ are associated to edges. Let us denote by $\nu$ their common law. We refer the reader to the recent review paper on first passage percolation by Damron et al~\cite{MR3729447}.
For each path $\gamma$ in the graph $(\Zd,E_d)$, we define
\begin{align}
\label{DEF:t}
t(\gamma)=\sum_{e\in\gamma} t_e, \quad \text{and} \quad \forall x,y \in \Zd, \; t(x,y)=\inf_{\gamma:x \to y} t(\gamma),
\end{align}
that can been seen as a random pseudo-distance on $\Zd$. Using Kingman's subadditive ergodic theorem allows to define
\begin{align}
\label{DEF:mu}
 \forall  x \in \Zd \quad \mu_{\nu}(x)&=\lim_{n\to +\infty} \frac{t(0,nx)}n,
\end{align}
where the limits hold almost surely and in $L^1$. The functional $\mu_{\nu}$ is homogeneous and subadditive, and can be extended to a symmetric semi-norm on $\Rd$. With some extra integrability assumption, we obtain
the analytic form of the asymptotic shape theorem:
\begin{align}
\label{PROP:FA}
\lim_{\|x\|\to +\infty} \frac{t(0,x)-\mu_{\nu}(x)}{\|x\|}&=0\quad \P \text{ a.s.}
  \end{align}
The subadditivity and the symmetries of the lattice imply quite simply that $\mu_{\nu}$ is a norm if and only if it $\mu_{\nu}((1,0,\dots,0))>0$ is strictly positive.
Moreover, it has long been known  (see for example Cox--Durrett~\cite{MR624685} or Kesten~\cite{kesten}) that 
$\mu_{\nu}$ is a norm  if and only  $\nu(\{0\})<p_c(\Zd)$, where $p_c(\Zd)$ is the critical percolation parameter for independent percolation on the edges of $\Zd$.

Our idea here is to find, in oriented percolation on $(\Zd,E)$, an analogous characterization of directions of percolation in terms of the semi-norm for an associated oriented first-passage percolation on $(\Zd,E)$.
Things are necessarily more intricate, since we saw that  for oriented percolation the critical probability may depend on the direction.

\subsection{Oriented percolation and first-passage percolation on $(\Zd,E)$}
We suppose that to each oriented bond $e\in E$ is associated a random variable $t_e$, the $(t_e)$'s being i.i.d.  integrable non-negative random variables, with $\nu$ as common distribution; we  denote by $p$ the probability $p=\P(t_e=0)=\nu(\{0\})$. 

In this section, we assume that the semi-group of $\Zd$ generated by $\text{Dir}$ is the whole set $\Zd$. Then, the graph $(\Zd,E)$ is transitive.

As in the classical setting, we can define the passage time of an oriented path as in \eqref{DEF:t}, use Kingman's subadditive ergodic theorem to define the associated functional $\mu_{\nu}$ as in \eqref{DEF:mu}, which is now positively homogeneous and subadditive but not necessarily symmetric. By sudadditivity, 
$$\forall x,y \in \Zd \quad |\mu_{\nu}(x+y)-\mu_{\nu}(x) | \le \|y\|_1 \max\{\mu(\epsilon e_i): \; 1 \le i  \le d, \; \epsilon \in\{0,1\} \}.$$
Thus $\mu_{\nu}$ can be extended in the usual way to a non-symmetric semi-norm on $\Rd$. Finally, 
we  get, under some extra integrability assumption, the analytic form of the asymptotic shape theorem as in \eqref{PROP:FA}. 

Our hope is to characterize the directions of percolations in $(\Zd,E)$ when edges are open with probability $p$, i.e. the $u\in \Rd$ such that 
$$\displaystyle D_u(0)=\sup_{y\in C_+(0)} \langle y,u\rangle=+\infty$$
with the help of the semi-norm $\mu_{\nu}$ for some law $\nu$  for the passage times of the edges. 
Since the only relevant parameter here is $\nu(\{0\})=p$, we take from now on 
$$\nu_p=p\delta_0+(1-p)\delta_1; $$
we denote by $\mu_p$ the associated semi-norm on $\Rd$  and we set
$$A_{p}=\{x\in\Rd: \; \mu_{p}(x)\le 1\},$$
which is a closed and convex set, but not necessarily bounded. We thus need some basics in the theory of convex sets.

\subsection{Convex sets} 
As $A_p$ is closed and convex, we can associate to $A_{p}$ two non-empty closed convex cones:

$\bullet$ The recession cone\footnote{sometimes called characteristic cone or asymptotic cone} of $A_{p}$ is
$$\displaystyle 0^+(A_{p})=\{u\in\Rd: \; A_{p}+\R_+u\subset A_{p}\}= \{x\in\Rd:\; \mu_{p}(x)=0\}.$$

$\bullet$ The barrier cone of $A_{p}$ is
$$\displaystyle \text{Bar}(A_{p})=\{u\in\Rd: \;  \miniop{}{\sup}{x\in A_{p}} \langle x,u\rangle <+\infty\}=\{x\in\Rd:\; b_p(x)>0\},$$
where $b_p(u)=\inf\{\mu_{p}(x):\;x \in \Rd \text{ such that } \langle u,x\rangle=1\} $.

The polar cone of a closed non-empty convex cone $C$  is defined by
$$C^{\circ}=\{u\in\Rd:\; \forall x\in C\quad  \langle x,u\rangle\le 0\}.$$
The map $C\mapsto C^{\circ}$ is an involutive map in the set of closed non-empty convex cones. Note also that  $C\cap C^{\circ}=\{0\}$.
Here, $0^+ (A_{p})$ is the polar cone associated to $\text{Bar}(A_{p})$ (see Rockafellar~\cite{MR0274683} Corollary 14.2.1 p 123). In other words, 
 characterizing the directions $x\in\Rd$ such that $\mu_{p}(x)=0$ is equivalent to characterizing the  directions $y\in\Rd$ such that  $b_p(y)>0$.

\subsection{Results}

Let us define, for $p\in[0,1]$,
\begin{align*}
\mathrm{BG}(p)=\left\{u \in \Rd: \;\P_p \left( \sup_{y\in C_+(0)} \langle y,u\rangle=+\infty \right)=0\right\}.
\end{align*}
Note that $\mathrm{BG}(p)$ is non-increasing in $p$. The set $\mathrm{BG}(p)$ collects the directions in which the growth of the cluster issued from $0$ is bounded. 
It is thus natural to make the following conjecture:
\begin{conjecture}
$\forall p\in [0,1] \quad  \mathrm{Bar}(A_{p})= \mathrm{BG}(p).$
\end{conjecture}

For the moment, we only manage to prove the following result: 

\begin{theo} For every $p \in [0,1]$, 
$$\mathrm{int}(\mathrm{Bar}(A_{p}))\subset \mathrm{BG}(p) \qquad \text{and} \qquad  \cup_{q>p} \mathrm{int}(\mathrm{BG}(q)) \subset \mathrm{Bar}(A_p).$$
\end{theo}
This result will be a direct consequence of corollaries \ref{coco1} and \ref{untheoqdmm}.

\medskip

As in the classical setting, we can describe the asymptotic
behavior of the point-to-hyperplane passage times with $\mu_p$. For $u \in\Rd\backslash\{0\}$ and $n \ge 0$, set
\begin{align*}
  H_n(u)&=\{x\in\Rd:\;\langle x, u\rangle\ge n\} \quad \text{and} \quad t(0,H_n(u))=\inf_{x \in H_n(u)} t(0,x). 
\end{align*}

%
%

\begin{theo}
\label{shapedir2}
For each $u\in\Rd$ which is not at the boundary of $A_{p}$, we have the almost sure convergence: 
$$ \lim_{n \to +\infty}  \frac{t(0,H_n(u))}{n}=  b_p(u) .$$
\end{theo}

\begin{proof}
As in the unoriented case, it will follow from the analytic form of the shape theorem. However, the existence of directions for which $\mu_p$  vanishes requires some attention. 

$\bullet$ Let $L>b_p(u)$. There exists $x \in \Rd$ with $\langle u,x\rangle=1$ and $\mu_{p}(x)\le L$. \\ 
For $n \ge 1$, denote by $x_n$ one vertex in $H_n(u)$ which is  the closest to $nx$. \\
Then $\mu_p(x_n) \le n\mu_p(x) +O(1)$.\\
Since $t(0,H_n(u))\le t(0,x_n)$, we have $\displaystyle \limsup \frac{t(0,H_n(u))}{n}\le\limsup   \frac{ \mu_{p}(x_n)}n \le \mu_p(x)\le L$. \\
Letting $L$ go to $b_p(u)$, we obtain  that $\displaystyle \limsup \frac{t(0,H_n(u))}{n}\le b_p(u)$.

$\bullet$ If $u \not \in \mathrm{Bar}(A_{p})$, then $b_p(u)=0$ and the desired convergence is clear.

$\bullet$ If $u \in \mathrm{int}(\mathrm{Bar}(A_{p}))$, there exists $\epsilon>0$ such that the open ball centered in $u$ with radius $\epsilon$ is included in $\mathrm{Bar}(A_{p})$; moreover, $b_p(u)>0$. By contradiction, assume that there exists $\ell \in(0,b_p(u))$ such that 
$$\miniop{}{\liminf}{n\to +\infty} \frac{t(0,H_n(u))}{n}\le\ell< b_p(u).$$
Then, one can build an infinite increasing sequence integers $(n_k)$ and sites  $(x_k)$ such that $t(0,x_k)\le\ell n_k$ and $\langle u,x_k\rangle = n_k+O(1)$.
By a compactness argument, we can assume that $\frac{x_k}{\|x_k\|}\to x$.
Then, $\frac{n_k}{\|x_k\|}=\langle \frac{x_k}{\|x_k\|},u\rangle +O(1/ \|x_k\|)\to \langle x,u\rangle$.
By the asymptotic shape theorem,  $\frac{t(0,x_k)}{\|x_k\|}$ tends to $\mu_{p}(x)$, and we get the inequality
$$\mu_{p}(x)\le\ell {\langle u,x\rangle}.$$
Assume that ${\langle u,x\rangle}=0$, then $\mu_{p}(x)=0$, so $x\in 0^+(A_{p})$. But $\mathrm{Bar}(A_{p})$ is the polar cone of $ 0^+(A_{p})$: by definition of $\epsilon$, $u+\epsilon x/2 \in \mathrm{Bar}(A_{p})$, so $0 \ge \langle u+\epsilon x/2,x\rangle = \epsilon/2$, which is a contradiction. \\
So assume that ${\langle u,x\rangle}\ne 0$:  we can define $\tilde x=\frac{x}{\langle u,x\rangle}$ and then $\langle u,\tilde{x}\rangle=1$ and $\mu_{p}(\tilde x)\le\ell$, which contradicts the definition of $b_p(u)$.
\end{proof}

\begin{coro} \label{coco1}
$\mathrm{int}(\mathrm{Bar}(A_{p}))\subset \mathrm{BG}(p).$
\end{coro}

\begin{proof}
Assume that $u\not\in\mathrm{BG}(p)$. Then, $\theta_u(p)>0$.
On the event $$\displaystyle  \sup_{x\in C_+(0)} \langle x,u\rangle=+\infty,$$for each $n\ge 1$, one can find $x_n\in C_+(0)$,  with $\langle x_n,u\rangle\ge n$. Then,  $x_n\in H_n(u)$ and $t(0,H_n(u))=0$. We then apply Theorem~\ref{shapedir2}.
\end{proof}

\begin{theo}
\label{exposoupc} 
Fix $u \in \Rd \backslash \{0\}$ such that $\miniop{}{\liminf}{x\to u} p_c(x)>0$
and fix $p$ such that $0<p<\miniop{}{\liminf}{x\to u} p_c(x)$.
There exist constants $A,B,\kappa>0$ such that
$$\forall n \ge 0 \quad \P(t(0,H_n(u))\le \kappa n)\le Ae^{-B n}.$$
\end{theo}

\begin{proof}
The idea is close to the one used by  Grimmett and Kesten~\cite{grimmett-kesten} to obtain large deviations inequalities for first-passage percolation: along an optimal path from $0$ to $H_n(u)$, we expect to find a number proportional to $n/N$ of disjoint streches whose increase in the $u$-direction is at least $N$. However, as $p<p_c(u)$, the first point of Theorem \ref{thm:perco} ensures that 
$$\P_p \left(\sup_{y\in C_+(0)}\langle y,u\rangle \ge N \right)$$
decreases exponentially fast with $N$, so with high probability, streches whose increase in the $u$-direction is at least $N$ have to use edges with passage time $1$, and should globally contribute to an amount of time $\kappa n$ for some small $\kappa>0$. A renormalisation argument allows to make all this accurate.

However, we did not manage to implement the renormalisation argument under the assumption $p<p_c(u)$, and we rather work under the stronger assumption
$$p<\miniop{}{\liminf}{x\to u} p_c(x).$$

  1. We can assume without loss that $\|u\|_1<1$.  
Then, we can find $\delta\in\R^2 \backslash \{0\}$ with $\langle u,\delta\rangle=0$ such that $v=u+\delta$ and $w=u-\delta$ satisfy $p<\min(p_c(v),p_c(w))$, $\|v\|_{1}<1$ and $\|w\|_{1}<1$.  By construction, $u=\frac{v+w}2$. 
We define the following set
$$T=\left\{x\in\R^2: \;\langle x,v\rangle \le 1, \; \langle x,w\rangle \le 1, \; \langle x,u\rangle \ge -10\right\}.$$
We can easily check that $T$ is bounded and thus is a triangle: for any $x \in T$,
\begin{equation}
\label{EQ-TB}
 -10\le \langle x,u\rangle\le 1,\quad -11\le \langle x,\delta\rangle\le 11, \quad \text{ and } \quad \|x\|_2\le C=\sqrt{\frac{100}{\|u\|_2^2}+\frac{121}{\|\delta\|_2^2}}.
 \end{equation}
As $\sup \{\langle x,v\rangle: \; x \in [-1,1]^2\}=\|v\|_1$, we also check that $[-1,1]^2 \subset T$, and we set
\begin{align*}
  \theta 
  &=\max(\|v\|_{1},\|w\|_{1})<1.  
\end{align*}

2. For an integer $N \ge 4$, we partition $\Z^2$ into boxes $(B_N(k))_{k \in \Z^2}= (2Nk+\{-N,\dots, N-1\}^2)_{k \in \Z^2}$. We set $B_N=B_N(0)$, and $T_N$ is the image of $T$ by the dilatation with ratio $N$. Note that $B_N$ is included inside $T_N$. We then define naturally the translated triangles  $T_N(k)=2Nk+ T_N$. 

Consider now a path $\gamma$ from $0$ to $H_n(u)$. As in Grimmett-Kesten~\cite{grimmett-kesten}, we now associate to this path a squeleton $\Gamma=(i_0, i_1, \dots,i_\ell)$ of distinct $N$-boxes and a sequence $(b_0, b_1, \dots, b_\ell)$ of sites such that, except for the last point, $b_k \in B_N(i_k) \subset  T_N(i_k)$, in the following manner. Set $i_0=0$ and $b_{0}=0$.
Suppose $i_0,\dots,i_n$, $b_0,\dots,b_n$ have been defined.

\def\fin{\ell}

\begin{itemize}
\item If the last point $\gamma_{\mathrm{last}}$ of $\gamma$ belongs to $T_N(i_n)$, then we  end the process and set $\ell=n$. 
\item Otherwise,
  let $b_{n+1}$ be the first point of the path that is outside $T_N(i_n)$ and define $i_{n+1}$  as the only index such that $b_{n+1}\in B_N(i_{n+1})$. We also a crossing type for $b_{n+1}$:
  
  \begin{itemize} \item if $\langle b_{n+1}-2Ni_n,v\rangle>N$ or $\langle b_{n+1}-2Ni_n,w\rangle>N$,  we say that the crossing type of $i_n$ is  $\text{up}$;
  \item otherwise $\langle b_{n+1}-2Ni_{n},u\rangle<-10N$ and we say that the crossing type of $i_n$ is $\text{down}$.
  \end{itemize}

 \end{itemize} 
We then remove the loops from this sequence, and we obtain, by relabeling the coordinates of the remaining $N$-boxes if necessary,  the squeleton  $\Gamma=(i_k)_{0 \le l \le \ell}$ of the path $\gamma$, see Grimmett-Kesten~\cite{grimmett-kesten} for details. We denote by $I_{\text{up}}(\gamma)$ and $I_{\text{down}}(\gamma)$  the number of crossings of the squeleton that are of the respective types $\text{up}$ and $\text{down}$. Note that $I_{\text{up}}+I_{\text{down}}(\gamma)=\ell$.
Let us now establish rough bounds for $I_{\text{up}}$ and $I_{\text{down}}$ by using the following decomposition
$$\gamma_{\textrm{last}}=(\gamma_{\textrm{last}}-b_\ell)+\sum_{k=0}^{\fin-1} (b_{k+1}-2Ni_k)+\sum_{k=0}^{\fin-1}(2Ni_k-b_k).$$
We have the following estimates:
\begin{itemize}
\item $\gamma_{\textrm{last}}$ and $b_\ell$ are both in $T_N(i_\ell)$, so with \eqref{EQ-TB}, $\langle \gamma_{\textrm{last}}-b_\ell,u \rangle \le 11N$.
\item For any $k \in \{1, \dots, \ell\}$,  let $a_k$
 be the last point of the path $\gamma$ before $b_k$ to be in $T_N(i_{k-1})$. As $a_k \in T_N(i_{k-1})$, with \eqref{EQ-TB} we have $\langle a_{k}-2Ni_{k-1},u\rangle\le N$. As $(a_k,b_k)$ is an edge, $\langle b_{k}-2Ni_{k-1},u\rangle\le N+K$, where 
 $$K=\max_{e \in\text{Dir}}\langle e,u \rangle >0.$$
 \item If $i_k$ is of type down, $\langle b_{k}-2Ni_{k-1},u\rangle\le -10N$.
 \item For any $k \in \{0, \dots, \ell-1\}$, $b_k \in B_N(i_{k})$, thus $\langle 2Ni_k-b_k, u \rangle \le N\|u\|_1\le N$.
\end{itemize}
As $\gamma$ is a path from $0$ to $H_n(u)$, this leads, for any fixed $N\ge K$, to
\begin{align*}
n & \le \langle \gamma_{\textrm{last}}, u \rangle  \\
& \le 11N + (N+K)I_{\text{up}}(\gamma)-10N I_{\text{down}}(\gamma)+N(I_{\text{up}}(\gamma)+I_{\text{down}}(\gamma)) \\
&  \le 11N+3N( I_{\text{up}}(\gamma)-3 I_{\text{down}}(\gamma)).
 \end{align*}
 From this, we first deduce that  for every $n$ large enough,
\begin{equation}
\label{EQ:upn}
\ell \ge I_{\text{up}}(\gamma) \ge \frac{n}{3N} -\frac{11}3 \ge \frac{n}{4N}.
\end{equation}
 And we also see that $3N( I_{\text{up}}(\gamma)-3 I_{\text{down}}(\gamma)) \ge n-11N \ge 0$ for every $n$ large enough, so 
 $3I_{\text{down}}(\gamma)\le I_{\text{up}}(\gamma)$ and 
 \begin{equation}
 \label{EQ:upell}
 I_{\text{up}}(\gamma)\ge\frac{3}4(I_{\text{up}}(\gamma)+I_{\text{down}}(\gamma))=\frac34 \ell.
 \end{equation}

\medskip
3. For $k\in\Z^2$, we say that the box $B_N(k)$ is good if for each $x\in B_N(k)$, 
$$\max\left( \sup_{y\in C_+(x)}\langle y,v\rangle, \sup_{y\in C_+(x)}\langle y,w\rangle \right)<N(1-\theta)$$ and that $B_N(k)$ is bad otherwise. 
As $p<\min(p_c(v),p_c(w))$, by the first point of Theorem \ref{thm:perco}, there exists $\alpha>0$ such that for every $n\ge 1$,
$$\P_p \left(\sup_{y\in C_+(0)}\langle y,v\rangle \ge n \right)\le e^{-\alpha n} \quad \text{ and } \quad \P_p \left(\sup_{y\in C_+(0)}\langle y,w\rangle \ge n \right) \le e^{-\alpha n}.$$
Thus, for every $N \ge 1$, $\P_p \left(B_N\text{ is bad}\right)\le 2(2N+1)^2 e^{-\alpha (1-\theta)N}$ and 
\begin{align*}
 \lim_{N \to +\infty} \P_p \left( \text{$B_N$ is good} \right)=1.
\end{align*}
Let us also denote by $I_{\text{up}}^G(\gamma)$  the number of boxes that belong to the squeleton, whose associated crossing is of type up, and that are also good. Assume that $i_k$ is the index of a crossing of type up, and for instance that
$\langle b_{k+1}-2Ni_k,v\rangle>N$. Then, as $b_k \in  B_N(i_k)$, with \eqref{EQ-TB}, 
\begin{align*}
\langle b_{k+1}-b_k,v  \rangle & = \langle b_{k+1}-2Ni_k,v  \rangle - \langle b_k-2Ni_k,v  \rangle 
 \ge N(1-\theta).
\end{align*}
If moreover the box $B_N(i_k)$ is good, then
$\sup_{y\in C_+(b_k)}\langle y,v\rangle < N(1-\theta)$, and this implies that the portion of the path $\gamma$ between $b_k$ and $b_{k+1}$ uses at least one edge with passage time $1$; and the same is true if the crossing is up because $\langle b_{k+1}-2Ni_k,w\rangle>N$. So
\begin{equation}
\label{EQ:tgamma}
t(\gamma) \ge I_{\text{up}}^G(\gamma).
\end{equation}
From now on, we will denote it as $I^{G}_{\text{up}}(\Gamma,\varepsilon)$ the number of good boxes associated to the couple formed by a squeleton and a sequence of up/down status for its crossings. 

\medskip
4. Remember that $C$ is defined in \eqref{EQ-TB}. We now fix the last parameters: choose $\alpha>0$ and $\rho \in (0,1)$ such that
$$\beta=8(C+1)^2(\rho e^{-\alpha}+(1-\rho))^{3/4}<1.$$
As the states of the boxes are identically distributed and only locally dependant and $\lim_{N \to +\infty} \P_p \left( \text{$B_N$ is good} \right)=1$, we can use the Liggett-Schonnmann-Stacey coupling result~\cite{LSS}: there exists $ N$ large enough such that the field of the states of the boxes $(B_N(k))_{k \in \Z^2}$ stochastically dominates a product of Bernoulli laws with parameter~$\rho$.
We then fix $\kappa>0$ small enough to have
$$e^{\alpha \kappa}\beta^{\frac1{4 N}}<1.$$
Thanks to \eqref{EQ:upn}, \eqref{EQ:upell} and \eqref{EQ:tgamma}, we have
\begin{align*}
& \P_p \left( t(0,H_n(u))\le \kappa n \right) \\
& \le  \P_p \left( \begin{array}{c}
\exists \ell \ge \frac{n}4, \;   \exists \Gamma=(\Gamma_k)_{0 \le k \le \ell} \text{ squeleton}, \; \exists (\varepsilon_k)_{0\le k<\ell}\in\{\text{up},\text{down}\}^\ell: \\
  I_{\text{up}}(\Gamma,\varepsilon)\ge \frac34 \ell, \; I_{\text{up}}^G(\Gamma,\varepsilon)  \le \kappa n \end{array} \right)
\end{align*}
Because of \eqref{EQ-TB}, for a fixed $\ell \ge n/4$, 
there are at most $(8(C+1)^2)^{\ell}$ couples $(\Gamma,\varepsilon)$ with length $\ell$.
If such a couple satisfies $I_{\text{up}}(\Gamma,\varepsilon)\ge \frac34\ell$, then
$I_{\text{up}}^G(\Gamma,\varepsilon)$ stochastically dominates a variable $S$ with binomial law with parameters $(\lceil \frac34\ell \rceil,\rho)$, so 
\begin{align*}
  \P_p(I_{\text{up}}^G(\Gamma,\varepsilon)  \le \kappa n)&\le \P(S\le \kappa n)=\P(e^{-\alpha S}\ge e^{-\alpha\kappa n})\\
  &\le e^{\alpha\kappa n}\E(e^{-\alpha S})\le e^{\alpha\kappa n}(\rho e^{-\alpha}+1-\rho)^{\frac34\ell}.
\end{align*}
So we obtain:
\begin{align*}
  \P_p \left( t(0,H_n(u))\le \kappa n \right) & \le \sum_{\ell\ge \frac{n}{4N}} (8(C+1)^2)^{\ell} e^{\alpha\kappa n}(\rho e^{-\alpha}+1-\rho)^{\frac34\ell}\\
  &\le e^{\alpha \kappa n} \sum_{\ell\ge \frac{n}{4N}} \beta^{\ell}=\frac1{1-\beta}(e^{\alpha \kappa}\beta^{\frac1{4N}})^n,
\end{align*}
which ends the proof. 
\end{proof}

\begin{coro} \label{coco2}
\label{hb}For each $u\in\Rd\backslash\{0\}$,
$p<\liminf_{x\to u}p_c(x)\Longrightarrow b_p(u)>0.$
\end{coro}
\begin{proof}
Suppose $p<\liminf_{x\to u}p_c(x)$. By Theorem~\ref{exposoupc}, there exist $c,\alpha>0$ such that for each $n\ge 1$, $\P_p(t(0,H_n(u))\le cn)\le e^{-\alpha n}$.
Then, with the Borel--Cantelli lemma and Theorem~\ref{shapedir2}, we get $b_p(u)\ge c$.
\end{proof}


\begin{coro}\label{untheoqdmm}
$\displaystyle \cup_{q>p} \mathrm{int}(\mathrm{BG}(q)) \subset \mathrm{Bar}(A_p)$.
\end{coro}

\begin{proof}
Consider $u\in \displaystyle \cup_{q>p} \text{int}(\mathrm{BG}(q))$:
there exists $q>p$ such that $u\in \text{int}(\mathrm{BG}(q))$, which means
that there exists $\delta>0$, with $B(u,\delta)\subset \mathrm{BG}(q)$.
For each $x\in B(u,\delta)$, we have $\theta_x(q)=0$ and $p_c(x)\ge q$.
This implies that $\liminf_{x\to u}p_c(x)\ge q>p$.
We conclude with Corollary~\ref{hb}.
\end{proof}


\def\refname{References}
\bibliographystyle{plain}
\bibliography{percoo-fpp}


\end{document}